\documentclass[a4paper]{article}
\usepackage{amsmath,amsfonts,amssymb}
\usepackage{graphicx}
\usepackage{eucal}
\usepackage{mathrsfs}
\usepackage{theorem}
\usepackage{pifont}
\usepackage {epsfig}
\usepackage {graphicx}
\newcommand{\Integer}{\mathbb{Z}}
\newcommand{\Natural}{\mathbb{N}}
\newcommand{\Complex}{\mathbb{C}}

\newcommand{\todo}[1]{{\sffamily To do:}}

\newtheorem{theorem}{Theorem}

\newtheorem {lemma}{Lemma}
\newenvironment{proof}{{\flushleft \emph{Proof}:}}{\ding{110}}
\newenvironment{proofmain}{{\flushleft \emph{Proof of Theorem \ref{main}}:}}{\ding{110}}

{\theorembodyfont{\rmfamily}

}

\title{Gambler's ruin probability - a general formula}
\author{Guy Katriel\footnote{Email: katriel@braude.ac.il}\\ Department of Mathematics, ORT Braude College,\\ Karmiel, Israel\\}

\date{}
\begin{document}
\maketitle

\begin{abstract}
We derive an explicit formula for the probability of ruin of a gambler playing against an infinitely-rich adversary, when the games have
payoff given by a general integer-valued probability distribution.
\end{abstract}

\section{Introduction}

We consider the classical Gambler's Ruin Problem. A gambler, who starts with an initial wealth $M$ ($M\in \Natural$), plays a series of games,
where in game $t$ the gambler's payoff is a random variable $X_t$ (independent for different values of $t$) with
$$P(X_t=k)=p_k,\;\;k\in \Integer,\;\;-\nu\leq k<\infty.$$
$\nu>0$ is thus the maximal possible loss, and we assume $p_{-\nu}\neq 0$.
The series of games proceeds until the gambler's wealth is less than $\nu$, in which case the gambler must stop playing, and we say that the
gambler is ruined. We are interested in the probability of ruin $P_{ruin}$, which depends on the payoff distribution $\{p_k\}_{k=-\nu}^\infty$ and on
the initial wealth $M$. It is well-known that if the expected value of $X_t$ is non-positive then $P_{ruin}=1$, so that we can
assume that we are dealing with a favorable game
\begin{equation}\label{pos}E(X_t)=\sum_{k=-\nu}^\infty kp_k>0.\end{equation}
Another trivial case is when the gambler's initial wealth is less than $\nu$, so that ruin occurs immediately, so we henceforth assume
$M\geq \nu$.

We will derive a formula for $P_{ruin}$, which, surprisingly, seems not to be available in the literature, except for some very special cases (see discussion below).

We define the generating function
$$p(z)=\sum_{k=-\nu}^\infty p_k z^k.$$
Note that, since $\sum_{k=-\nu}^\infty p_k=1$, the series $p(z)$ defines a meromorphic function in the unit disk $|z|<1$ of the complex plane, with a unique pole, of order $\nu$, at $z=0$.

For integers $n>0,r\geq 0$, the complete symmetric polynomial of order $r$ in the variables $z_1,\cdots,z_n$ is defined as the sum of all products
of the variables $z_1,\cdots,z_n$ of degree $r$, that is
$$\Phi_{n,r}(z_1,\cdots,z_n)=\sum_{i_j\geq 0,\;i_1+\cdots+i_n=r}\prod_{j=1}^n z_j^{i_j}.$$

\begin{theorem}\label{main} The equation $p(z)=1$ has $\nu$ solutions (counting multiplicities) in the unit disk $|z|<1$ of the complex plane, which we denote by
$\eta_j$ ($1\leq j\leq \nu$).
The ruin probability is given by
\begin{equation}\label{mr}P_{ruin}(M)=\sum_{n=1}^\nu  \Phi_{n,M-n+1}(\eta_1,\cdots,\eta_n) \prod_{j=1}^{n-1} (1-\eta_j).\end{equation}
When the roots $\eta_1,\cdots,\eta_\nu$ are distinct, one can use the alternative expression
\begin{equation}\label{mra}P_{ruin}(M)=\sum_{j=1}^\nu \eta_j^M \prod_{i\neq j}\frac{1-\eta_i}{\eta_j-\eta_i}.\end{equation}
\end{theorem}
We note that empty products, as occurs in (\ref{mr}) for $n=1$, are to be interpreted as $1$. The two expressions (\ref{mr}),(\ref{mra}) are algebraically equivalent, but
each has its advantages. (\ref{mr}) allows for the case of multiple roots, and shows that the expression on the right-hand side of (\ref{mra}) is in fact a polynomial in $\eta_1,\cdots,\eta_\nu$. (\ref{mra}) clearly shows that  $P_{ruin}(M)$ is is exponentially decaying as a function of $M$ (a well-known fact), with the
rate of decay determined by the root $\eta_j$ with maximal absolute value.

The special case of the gambler's ruin problem in which
\begin{equation}\label{sc0}p_{-1}>0,\;p_1>0,\;\;\;p_k=0, \;\;|k|\geq 2,\end{equation}
is treated in every textbook of probability theory, and it is shown that
\begin{equation}\label{sc}P_{ruin}(M)=\Big(\frac{p_{-1}}{p_1}\Big)^M\end{equation}
(note that the assumption (\ref{pos}) implies $p_{-1}<p_1$). (\ref{sc}) is obtained as an especially simple case of (\ref{mr}), because
when (\ref{sc0}) holds we have $\nu=1$, $p(z)=p_{-1}z^{-1}+p_0+p_1 z$, so that $\eta_1=\frac{p_{-1}}{{p_1}}$, and (\ref{mr}) gives
$$P_{ruin}(M)= \Phi_{1,M}(\eta_1)=\eta_1^M=\Big(\frac{p_{-1}}{p_1}\Big)^M.$$
In general,  computing $P_{ruin}$ using (\ref{mr}) requires numerically solving $p(z)=1$  in the unit disk $|z|<1$.

The gambler's ruin problem, going back to Pascal and Huygens (see \cite{edwards} and \cite{ethier}, Sec. 7.5 for history), is usually considered together with a version in which the games also stop if the gambler's fortune exceeds a threshold $W$ (the gambler wins, or the casino is ruined). Our problem corresponds to $W=\infty$, and is often
called gambling against an infinitely rich adversary. The result (\ref{sc}) is usually obtained as a limit of an expression for finite $W$ when $W\rightarrow \infty$.

Feller (\cite{feller}, Sec. XIV.8) and Ethier (\cite{ethier}, Sec. 7.2) discuss the problem with an upper threshold $W$ in
the case of a payoff distribution $\{p_k\}_{k=-\nu}^{\mu}$, supported both from below and from above. 
They show that the ruin probabilities can be computed by solving a system of linear equations of size $W-\nu+1$ for $P_{ruin}(M)$ ($\nu\leq M\leq W$).
It is also possible to express $P_{ruin}(M)$ in the form
\begin{equation}\label{fel}P_{ruin}(M)=\sum_{j=1}^{\nu+\mu}a_j \lambda_j^M,\;\;\;\nu\leq M\leq W, \end{equation}
where $\lambda_j$ ($1\leq j\leq \nu+\mu$) are {\it{all}} the roots of $p(z)=1$,
with the coefficients $a_j$ determined by a linear system of equations of size $\nu+\mu$. Here again there is no general explicit expression for the solution,
so that efforts have been devoted to proving upper and lower bounds for $P_{ruin}$, either in terms of the largest root of the equation $p(z)=1$  \cite{ethier,feller}, or more recently in terms of the moments of the payoff distribution \cite{ethier1,hurlimann}.

One import of Theorem \ref{main} is that the problem with $W=\infty$ is simpler than the case $W<\infty$, in the sense that an explicit
formula for the ruin probability can be derived. Comparing (\ref{fel}) (for the case $W<\infty$) and (\ref{mra}) (for the case $W=\infty$),
we see that in both formulas we have linear combinations of the $M$-th powers of roots
of the equation $p(z)=1$, but in (\ref{fel}) the coefficients $a_j$ must be determined by solving a (potentially large) system
of linear equations, while in (\ref{mra}) everything is given explicitly. Note also that in (\ref{mra}) {\it{only}} those roots which are inside the unit disk are involved.
Our derivation does not use (\ref{fel}) or any other result for the case $W<\infty$ - we deal directly with the case $W=\infty$.
Our approach
relies on the construction of appropriate generating functions and their analysis.

A previous result which comes closest to the formula (\ref{mra}), in a very special case, was obtained by Skala \cite{skala}, who considered the case
$$p_{-\nu}>0,\;\;p_{\mu}>0,\;\;\;p_k=0,\;\;k\neq -\nu,\mu,$$
and showed that the formula (\ref{mra}) holds. Skala used the system of linear equations mentioned above for the case $W<\infty$, and a process
of going to the limit $W\rightarrow \infty$.

Let us remark that in some contexts `ruin' would be defined as the event in which the gambler's fortune drops to $0$ or less, rather than below $\nu$. Translating the results to this case is trivial: on the right-hand sides of (\ref{mr}),(\ref{mra}), $M$ will be replaced by $M+\nu-1$.

We conclude this introduction with a numerical example. Let us consider a game in which participation costs $\$ \nu$ and with a prize
which is Poisson distributed with mean $\nu+\epsilon$ ($\epsilon>0$). Then $p(z)=z^{-\nu}e^{(\nu+\epsilon)(z-1)}$. We take $\nu=3$ and $\epsilon=0.01$.
The solutions of $p(z)=1$ in the open unit disk (found numerically using MAPLE) are then given by: $z_1=0.993362,z_2=-.202699+.220049 i, z_3=\bar{z}_2$. Using (\ref{mra}) we obtain
$$P_{ruin}(3)=0.9900,\;\;P_{ruin}(10)=0.9456,\;\;P_{ruin}(50)=0.7245,$$
$$P_{ruin}(100)=0.5193,\;\;P_{ruin}(200)=0.2668,\;\;P_{ruin}(500)=0.0361.$$

\section{Proof of the ruin probability formula}
Defining $S_t$ to be the gambler's wealth after the $t$-th game, we have $S_0=M$, and, for all integer $t\geq 0$
\begin{equation}\label{ds}S_{t+1}=\left\{
        \begin{array}{ll}
          S_t+X_t, & S_t\geq \nu \\
          S_t, & 0\leq S_t<\nu
        \end{array}
      \right..
\end{equation}
We define the sequence $S_t$ so that if ruin occurs at some time $t_0$, then $S_t=S_{t_0}\in \{0,1,\cdots,\nu-1\}$ for all $t\geq t_0$, or, in other
words, $\{0,1,\cdots,\nu-1\}$ are absorbing states of the Markov process $S_t$.
Thus the probability of ruin at or before time $t$ is given by $\sum_{k=0}^{\nu-1} P(S_t=k)$, and the probability of ultimate ruin is
\begin{equation}\label{pr}P_{ruin}=\lim_{t\rightarrow \infty}\sum_{k=0}^{\nu-1}P(S_t=k).\end{equation}
For later use, we remark that another way to express (\ref{ds}) is
\begin{equation}\label{ds1}
P(S_{t+1}=k)=\left\{
               \begin{array}{ll}
                 \sum_{l=-\nu}^{k-\nu} p_l \cdot P(S_t=k-l), & k\geq \nu\\
                 P(S_t=k)+\sum_{l=-\nu}^{k-\nu} p_l \cdot P(S_t=k-l), & 0\leq k<\nu
               \end{array}
             \right..
\end{equation}

With each of the random variables $S_t$ we associate its generating function
\begin{equation}\label{gf}f_t(z)=E(z^{S_t})= \sum_{k=0}^\infty P(S_t=k)z^k.\end{equation}
Note that $f_0(z)=z^M$. Since the coefficients $P(S_t=k)$ are non-negative and their sum is $1$, the series (\ref{gf}) converges
uniformly in the closed unit disk $|z|\leq 1$ of the complex plane, and thus defines a continuous function in the closed unit disk, which is
holomorphic in the open unit disk, with
\begin{equation}\label{bf} |z|\leq 1\;\;\Rightarrow\;\;|f_t(z)|\leq 1.\end{equation}

Denoting by ${\cal{P}}$ the vector space of all power series, $f(z)=\sum_{k=0}^\infty c_k z^k $, we define the truncation operator $T: {\cal{P}}\rightarrow {\cal{P}}$,
$$T\left[\sum_{k=0}^\infty c_k z^k \right]=\sum_{k=0}^{\nu-1} c_k z^k.$$

\begin{lemma}\label{bas} The sequence of polynomials (of degree $\nu-1$) $\{T[f_t]\}_{t=0}^\infty$
converges to a polynomial $Q(z)$
\begin{equation}\label{Q}Q(z)=\lim_{t\rightarrow \infty}T[f_t](z),\;\;\;\forall z\in \Complex,\end{equation}
and
\begin{equation}\label{pr1}P_{ruin}=Q(1).\end{equation}
\end{lemma}

\begin{proof} We have
\begin{equation}\label{y}T[f_t](z)=\sum_{k=0}^{\nu-1}P(S_t=k)z^k,\end{equation}
and since, for each $0\leq k\leq \nu-1$, $P(S_k=t)$ is increasing as a function of $t$ and bounded by $1$, we have
convergence of each of the coefficients of $T[f_t](z)$ as $t\rightarrow \infty$, so that the limiting degree-$(\nu-1)$ polynomial
exists, and (\ref{Q}) holds.
In view of (\ref{pr}), (\ref{Q}) and (\ref{y}), we have (\ref{pr1}).
\end{proof}

We will find $P_{ruin}$ by finding an explicit expression for $Q(z)$ and taking $z=1$. We note that the  coefficient of $z^k$ ($0\leq k \leq \nu-1$) in $Q(z)$ is the probability that ruin occurs with final fortune $k$.

We now derive a recursive formula relating $f_{t+1}$ to $f_t$.
\begin{lemma} \label{rr} For all $t\geq 0$,
\begin{equation}\label{rec}f_{t+1}(z)=T[f_t](z)+p(z)\Big[f_t(z)-T[f_t](z)\Big].\end{equation}
\end{lemma}

\begin{proof} We show that, for each integer $k\geq 0$, the coefficient of $z^k$ in the power series $f_{t+1}(z)$, is equal to
the coefficient of $z^k$ in the power series on the right-hand of (\ref{rec}). Using the notation $[z^k]g(z)$ to refer to the coefficient of $z^k$ in a
power series $g(z)$, we have
\begin{equation}\label{c1}[z^k]f_{t+1}(z)=P(S_{t+1}=k), \end{equation}
\begin{equation}\label{c2}[z^k](T[f_{t}](z))=\left\{
                                             \begin{array}{ll}
                                               0, & k\geq \nu\\
                                               P(S_t=k), & 0\leq k\leq \nu-1
                                             \end{array}
                                           \right.
, \end{equation}
\begin{equation}\label{c3}[z^k]\left( p(z)\Big[f_t(z)-T[f_t](z)\Big]\right)=\sum_{l=-\nu}^{k-\nu} p_l  \cdot P(S_t=k-l),
\end{equation}
and using (\ref{ds1}) we see that the right-hand side of (\ref{c1}) is equal to the sum of the right-hand sides of (\ref{c2}) and (\ref{c3}).
\end{proof}

Introducing the linear operator $L:{\cal{P}}\rightarrow {\cal{P}}$ defined by
\begin{equation}\label{L}L[f](z)=T[f](z)+p(z)\Big[f(z)-T[f](z)\Big],\end{equation}
Lemma \ref{rr} tells us that, for all $t\geq 0$,
\begin{equation}\label{it}f_{t+1}=L[f_t].\end{equation}

For each $w\in [0,1)$, we define the function
$$F_w(z)=\sum_{t=0}^\infty w^t f_t(z).$$
By (\ref{bf}), the above series converges uniformly in $|z|\leq 1$ for each $w\in [0,1)$, and since each $f_t$ is a continuous function in the closed unit disk, holomorphic in its interior,
$F_w$ also has these properties. We note also that the functions $w\rightarrow F_w(z)$ ($z$ fixed) are continuous, uniformly for $z$ in compact subsets of the unit disk,
which in particular implies that the mapping $w\rightarrow T[F_w](z)$ from $[0,1)$ to the space of degree-$(\nu-1)$ polynomials is continuous, a
fact that will be used below.

\begin{lemma}\label{llim} For any $z\in \Complex$ we have
\begin{equation}\label{lq}Q(z)=\lim_{w\rightarrow 1-}(1-w)T[F_w](z),\end{equation}
where $Q(z)$ is the polynomial defined by (\ref{Q}).
\end{lemma}

\begin{proof} Fix $z$. Set
$$c_t=T[f_t](z)-T[f_{t-1}](z),\;\;t\geq 1.$$
By Lemma \ref{bas}, we have
$$\sum_{t=1}^\infty c_t=-T[f_0](z)+\lim_{t\rightarrow \infty}T[f_t](z)=Q(z)-T[f_0](z).$$
Therefore Abel's theorem on power series (\cite{rudin} Theorem 8.2) implies that
\begin{equation}\label{abel}\lim_{w\rightarrow 1-} \sum_{t=1}^\infty c_t w^t=Q(z)-T[f_0](z).\end{equation}
For any $w\in [0,1)$ we have
\begin{eqnarray}\label{kl}(1-w)\sum_{t=0}^\infty w^tT[f_t](z)&=& T[f_0](z)+\sum_{t=1}^\infty w^t\Big[T[f_t](z)-T[f_{t-1}](z)\Big] \nonumber\\ &=&T[f_0](z)+\sum_{t=1}^\infty c_t w^t.\end{eqnarray}
Taking the limit $w\rightarrow 1-$ in (\ref{kl}), using (\ref{abel}), we get
$$\lim_{w\rightarrow 1-} (1-w)\sum_{t=0}^\infty w^t T[f_t](z) =Q(z).$$
Since
$$T[F_w](z)=\sum_{t=0}^\infty w^t T[f_t](z),$$
we have (\ref{lq}).
\end{proof}

In view of Lemma \ref{llim}, we now want to find explicit expressions for the polynomials $T[F_w](z)$ ($w\in [0,1)$).
We need the following result on the solutions of the equation
\begin{equation}\label{zz}p(z)=w^{-1}.\end{equation}

\begin{lemma}\label{zeros}
For any $w\in (0,1]$, (\ref{zz}) has $\nu$ roots (counted with multiplicities) $z=z_j(w)$ ($1\leq j\leq \nu$) in the open unit disk $|z|<1$.
\end{lemma}
\begin{proof}
We first derive an a-priori bound for all solutions.
Restricting the function $p(z)$ to the interval $(0,1]$, we have, using the fact that $p_k$ are non-negative, that $p(z)$ a convex function with $\lim_{z\rightarrow 0+}p(z)=+\infty$, $p(1)=1$, $p'(1)=E(X_t)>0$, so that
elementary calculus implies that there is a unique $z^*\in (0,1)$ with
\begin{equation}\label{te}p(z^*)=1,\;\;p(z)<1 \;{\mbox{for}} \;z\in (z^*,1).\end{equation}
We claim that any solution $|z|<1$
of $p(z)=w^{-1}$, $w\in (0,1]$, satisfies $|z|\leq z^*$. Indeed if $z$ satisfies $p(z)=w^{-1}$ then
$$p(|z|)= \sum_{k=-\nu}^\infty p_k |z|^k  \geq \Big|\sum_{k=-\nu}^\infty p_k z^k \Big| = |p(z)| = w^{-1} \geq 1,$$
which, by (\ref{te}), implies $|z|\leq z^*$.

Defining
$$h(z)=z^\nu p(z)=\sum_{k=0}^\infty p_{k-\nu}z^k,$$
which is continuous on the closed unit disk and holomorphic in the open unit disk,
we can write the equation $p(z)=w^{-1}$ in the form
\begin{equation}\label{aw}z^\nu=wh(z).\end{equation}

When $w=0$,  $z=0$ is a zero of multiplicity $\nu$ of (\ref{aw}).
Since, as shown above, all solutions of (\ref{aw}) in the unit disk satisfy the $|z|<z^*<1$, we conclude by the argument principle that (\ref{aw}) has
$\nu$ solutions for all $w\in [0,1]$.
\end{proof}

\begin{lemma} For any $w\in (0,1)$, $|z|\leq 1$,
\begin{equation}\label{tn} T[F_w](z)=\frac{1}{1-w}\sum_{n=1}^\nu  \Phi_{n,M-n+1}(z_1(w),\cdots,z_n(w)) \prod_{j=1}^{n-1} (z-z_j(w)),\end{equation}
where $z_j(w)$ ($1\leq j\leq \nu$) are the solutions of $p(z)=w^{-1}$ in the unit disk $|z|<1$, ordered arbitrarily.

If we further assume that $z_j(w)$ ($1\leq j\leq \nu$) are different from each other then we can also write
\begin{equation}\label{tna}T[F_w](z)=\frac{1}{1-w}\sum_{j=1}^\nu (z_j(w))^M \prod_{i\neq j}\frac{z-z_i(w)}{z_j(w)-z_i(w)}.\end{equation}
\end{lemma}

\begin{proof} Using the linearity of the operator $L$ defined by (\ref{L}), and the relation (\ref{it}), we have
$$L[F_w](z)=\sum_{t=0}^\infty w^tL[f_t](z)=\sum_{t=0}^\infty w^tf_{t+1}(z)=w^{-1}\sum_{t=1}^\infty w^tf_{t}(z)=w^{-1}[F_w(z)-f_0(z)],$$
which we can write as
$$F_w(z)=w L[F_w](z)+f_0(z),$$
or, using the definition of $L$,
$$F_w(z)=wT[F_w](z)+wp(z)\Big[F_w(z)-T[F_w](z)\Big]+f_0(z).$$
Solving for $F_w$, we have
\begin{equation}\label{fw}F_w(z)=\frac{w(1-p(z))T[F_w](z)+f_0(z)}{1-p(z)w}.\end{equation}
Since $F_w(z)$ is holomorphic for $|z|<1$, the numerator of the right-hand side of (\ref{fw}) must vanish
whenever the denominator does (otherwise we would have a pole), that is
$$|z|<1,\;1-p(z)w=0\;\;\Rightarrow\;\; w(1-p(z))T[F_w](z)+f_0(z)=0.$$
In other words, denoting by $z_j(w)$ ($1\leq j\leq \nu$) the solutions of (\ref{zz}) (recall Lemma \ref{zeros}) we have
$$w(1-p(z_j(w)))T[F_w](z_j(w))+f_0(z_j(w))=0,\;\;\;1\leq j\leq \nu,$$
and using $wp(z_j(w))=1$, we can rewrite this as
\begin{equation}\label{tt}T[F_w](z_j(w))=\frac{1}{1-w}f_0(z_j(w)),\;\;\;1\leq j\leq \nu.\end{equation}

We will now temporarily assume that $w$ is chosen so that $z_j(w)$ ($1\leq j\leq \nu$) are {\it{distinct}}. This holds for
all but a discrete set of values of $w\in [0,1)$, since if $p(z)=w^{-1}$ has a multiple root $z$ then $p'(z)=0$, so that $w^{-1}$ is a critical value of $p$, and
the set of critical values of a nonconstant holomorphic function is discrete.

$T[F_w](z)$ is a polynomial of degree $\nu-1$, and since (\ref{tt}) prescribes the values of
this polynomial at the $\nu$ points $z_j(w)$ ($1\leq j\leq \nu$), we can use Newton's interpolation formula \cite{phillips} to write this polynomial as
\begin{equation}\label{et}T[F_w](z)=\frac{1}{1-w}\sum_{n=1}^\nu f_0[z_1(w),\cdots,z_n(w)]\prod_{j=1}^{n-1} (z-z_j(w)),\end{equation}
where $f[z_1(w),\cdots,z_n(w)]$ is the $n$-th divided difference. We now use the following result (\cite{phillips}, Theorem 1.2.1): when $f_0(z)=z^M$ and
$0\leq n\leq M$,
\begin{equation}\label{phil}f_0[z_1,\cdots,z_n]=\Phi_{n,M-n+1}(z_1,\cdots,z_n).\end{equation}
Substituting (\ref{phil}) into (\ref{et}), we obtain (\ref{tn}).

We have proved (\ref{tn}) under the assumption that $z_j(w)$ are distinct, but we may now approximate any value of $w$ for which some $z_j(w)$
coincide by values of $w$ for which $z_j(w)$ are distinct, and use the continuity of both sides of (\ref{tn}) to conclude that it is valid
for all $w$.

If instead of using Newton's interpolation formula we use Lagrange's formula, we obtain (again under the assumption that $z_j(w)$ are distinct)
$$T[F_w](z)=\frac{1}{1-w}\sum_{j=1}^\nu f_0(z_j(w)) \prod_{i\neq j}\frac{z-z_i(w)}{z_j(w)-z_i(w)}.$$
Recalling that $f_0(z)=z^M$, we get (\ref{tn}). This formula, of course, does not make sense when $z_j(w)$ are not distinct.
\end{proof}

We are now ready to obtain the
\begin{proofmain}
Fixing any $|z|\leq 1$, (\ref{lq}) and (\ref{tn}) imply,
\begin{eqnarray}\label{ee1}Q(z)&=&\lim_{w\rightarrow 1-}(1-w)T[F_w](z)\\
&=&\lim_{w\rightarrow 1-}\frac{1}{1-w}\sum_{n=1}^\nu  \Phi_{n,M-n+1}(z_1(w),\cdots,z_n(w)) \prod_{j=1}^{n-1} (z-z_j(w))\nonumber. \end{eqnarray}
By Lemma \ref{zeros}, we have
\begin{eqnarray}\label{ee2}&&\lim_{w\rightarrow 1-}\frac{1}{1-w}\sum_{n=1}^\nu  \Phi_{n,M-n+1}(z_1(w),\cdots,z_n(w)) \prod_{j=1}^{n-1} (z-z_j(w))\\&=&\frac{1}{1-w}\sum_{n=1}^\nu  \Phi_{n,M-n+1}(\eta_1,\cdots,\eta_n) \prod_{j=1}^{n-1} (z-\eta_j),\nonumber\end{eqnarray}
where $\eta_j$ ($1\leq j\leq \nu$) are the roots of $p(z)=1$ in the open unit disk. Combining (\ref{ee1}) and (\ref{ee2}) we have
$$Q(z)=\frac{1}{1-w}\sum_{n=1}^\nu  \Phi_{n,M-n+1}(\eta_1,\cdots,\eta_n) \prod_{j=1}^{n-1} (z-\eta_j),$$
and setting $z=1$ and using (\ref{pr1}) gives (\ref{mr}).

If we assume that $\eta_1,\cdots,\eta_\nu$ are distinct, which implies that $z_1(w),\cdots,z_\nu(w)$ are distinct for $w$ sufficiently close to $1$,
then the same argument, using (\ref{tna}) instead of (\ref{tn}), leads to
$$Q(z)=\sum_{j=1}^\nu \eta_j^M \prod_{i\neq j}\frac{z-\eta_i}{\eta_j-\eta_i},$$
and substituting $z=1$ gives (\ref{mra}).
\end{proofmain}


\begin{thebibliography}{9}



\bibitem{edwards} A.W.F. Edwards, Pascal's problem: the ``gambler's ruin'', Int. Statistical Rev. {\bf{51}} (1983), 73-79.

\bibitem{ethier} S.N. Ethier, The Doctrine of Chances: Probabilistic Aspects of Gambling, Springer-Verlag, Berlin 2010.

\bibitem{ethier1} S.N. Ethier, D.Khoshnevisan, Bounds on gambler's ruin probabilities in terms of moments, Methodology
and Computing in Applied Probability {\bf{4}} (2002), 55-68.

\bibitem{feller} W. Feller, An Introduction to Probability Theory and its Applications, Vol. I, John Wiley \& Sons, New York 1968.

\bibitem{hurlimann} W. H\"urlimann, Improved analytical bounds for gambler's ruin probabilities, Methodology and Computing in Applied Probability {\bf{7}} (2005), 79-95.

\bibitem{phillips} G.M. Phillips, Interpolation and Approximation by Polynomials, Springer-Verlag, New York 2003.

\bibitem{rudin} W. Rudin, Principles of Mathematical Analysis, McGraw-Hill, New York 1976.

\bibitem{skala} H.L. Skala, An aspect of the gambler's ruin problem, Int. J. Math. Educ. Sci. Technnol {\bf{22}} (1991), 51-56.


\end{thebibliography}
\end{document}